\documentclass[13]{amsart}
\usepackage[utf8]{inputenc}
\usepackage[all]{xy}
\usepackage{tikz-cd}
\usepackage{bm}
\usepackage{enumerate}

\newcommand{\Maps}{\operatorname{Maps}}

\newcommand{\cc}{\mathrm{c}}

\newcommand{\Uone}{\operatorname{U}(1)}
\newcommand{\Rmod}{\mathbb{R}/\mathbb{Z}}

\newcommand{\ot}{\otimes}

\newcommand{\Mod}[1]{\ (\operatorname{mod}\ #1)}

\DeclareMathOperator{\cuad}{\operatorname{Quad}}

\DeclareMathOperator{\Hom}{Hom}

\newcommand{\B}{\mathcal{B}}

\numberwithin{equation}{section}

\newtheorem{theorem}{Theorem}[section]

\newtheorem{corollary}[theorem]{Corollary}

\newtheorem{proposition}[theorem]{Proposition}
\theoremstyle{definition}

\newtheorem{remark}[theorem]{Remark}

\newtheorem{definition}[theorem]{Definition}

\newcommand{\Fmatrix}[4]{F {\scriptscriptstyle \left[\begin{matrix} $#4$\\ $#1$,$#2$,$#3$ \end{matrix}\right]}}
\newcommand{\FuSp}[3]{\begin{bmatrix}$#3$ \\ $#1$,$#2$ \end{bmatrix}}

\begin{document}
\title[Solutions of the hexagon equation for abelian anyons]
{Solutions of the hexagon equation for abelian anyons}

\author{C\'esar Galindo}

\address{Departamento de Matem\'aticas\\Universidad de los Andes\\Carrera 1 N. 18A -10, Bogot\'a\\
Colombia}
\email{cn.galindo1116@uniandes.edu.co}

\author{ Nicol\'as Jaramillo Torres}

\address{Departamento de Matem\'aticas\\Universidad de los Andes\\Carrera 1 N. 18A -10, Bogot\'a\\
Colombia}
\email{n.jaramillo1163@uniandes.edu.co}

\thanks{C. G. was partially supported by the FAPA funds from vicerrectoria de investigaciones de la Universidad de los Andes.}

\begin{abstract}
We address the problem of determining the obstruction  to 
existence of solutions of the hexagon equation for abelian fusion rules and the classification of prime  abelian anyons.
\end{abstract}
\maketitle

\section{Introduction}
 
Anyons are two-dimensional particles which in contrast to boson or fermions satisfy exotic statistics.  The exchange of two identical anyons can in general be described by either abelian or non-abelian statistics. In the abelian case an exchange of two particles gives rise to a complex phase $e^{2\pi i \theta}$. Bosons and fermions correspond only to the phase changes $+1$  and $-1$ respectively. Particles with non-real phase change are considered anyons. In general,  the statistics of anyons is described by unitary operators acting on a finite dimensional degenerate ground-state manifold, \cite{kitaev2003fault}.

There has been increased interest in non-abelian anyons since they possess the ability to store, protect and manipulate quantum information \cite{kitaev2003fault,kitaev2006anyons,freedman2003topological,wang2010topological,pachos2012introduction}. In contrast, abelian anyons only seem good as quantum memory. Moreover, abelian anyons are interesting  for two reasons. First, they have simpler physical realizations than non-abelian anyons; and second, gauging a finite group of topological symmetries of an abelian anyon theory, when it possible, leads
to a new anyon theory that is in general is non-abelian, \cite{cui2015gauging}. 
Moreover, all concrete known examples of non-abelian anyon theories with integer global dimension are constructed from a gauging of an abelian anyon theory.

Mathematically speaking, an abelian anyon theory is a modular pointed category (\cite{drinfeld2010braided,etingof2005fusion}), and the latter comprise are the class of modular categories which are best understood.
Abelian anyons correspond to  triples $(A,\omega,c)$, where $A$ is a finite abelian group and  $(\omega, c)\in Z^3_{ab}(A,\Uone)$ is an abelian 3-cocycle. The set of modular categories up to gauge equivalence  with a fixed abelian group $A$ forms an abelian group denoted by $H^3_{ab}(A,\Uone)$ and called the \textit{third abelian cohomology group} of $A$. The groups $H^3_{ab}(A,B)$ were defined and studied by  Eilenberg and MacLane in \cite{eilenberg1,eilenberg2} for any  pair of abelian groups
\cite{eilenberg1,eilenberg2}. In this work, we address the problem of determining for an ordinary 3-cocycle $\omega \in Z^3(A,B)$ the obstruction to the existence of a map $c:A\times A\to B$ such that $(\omega,c)\in Z^3_{ab}(A,B)$. To that end, we construct a double complex associated to a finite abelian group and a map from the ordinary group cohomology to the total cohomology of the double complex.
We find several exact sequences involving $H^3_{ab}(A,B)$ and provide an explicit method for the construction of all possible abelian 3-cocycles. We finish the note with a reformulation of an old result of Wall \cite{MR0156890} and Durfee \cite{MR0480333} on the classification of indecomposable symmetric forms on finite abelian groups in terms of classification of prime abelian anyons.

The paper is organized as follows. In Section 2 we recall the definitions of group cohomology and abelian group cohomology. Section 3 contains a brief introduction to fusion algebras and the pentagon and hexagon equation. In section 4 we present the main results of the paper. We recall a theorem of  Eilenberg and MacLane  about the isomorphism between $H^3_{ab}(A,B)$ and $\cuad(A,B)$ (the group of all quadratic forms from $A$ to $B$). We also show that $\cuad(A,\Rmod)$ can be computed inductively from a decomposition of $A$ as direct sum of cyclic groups.  In this section we also define the obstruction  for the existence of solutions of the hexagon equation. Section 5 contains the classification of prime abelian anyon theories.

\smallbreak\subsubsection*{Acknowledgements} The authors are grateful to Paul Bressler, Richard Ng, Eric Rowell and Julia Plavnik for  useful discussions. This research was partially supported by the FAPA funds from Vicerrector\'{i}a de Investigaciones de la Universidad de los Andes.

\section{Preliminaries}

In this section we present some basic definitions of group cohomology and abelian group cohomology. A lot of this material can be found in \cite{eilenberg1} and \cite{eilenberg2}.  

We will denote by $\operatorname{U}(1)$ the group of
complex numbers of modulus $1$, which we will often write additively through the identification with $\mathbb{R}/\mathbb{Z}$.

Given an abelian group $A$ we will denote by $S^2(A), \wedge^2 A$ and $A^{\otimes 2}$ the second symmetric power, second exterior power and second tensor power of $A$, respectively. Here, we see $A$ as a $\mathbb{Z}$-module.

Given a group $G$ we will denote by $\widehat{G}$ to the abelian group of all linear character of $G$, that is $$\widehat{G}=\Hom(G,\Uone)=\Hom(G,\Rmod).$$

\subsection{Group cohomology}

We will recall the usual cocycle description of group cohomology associated to the normalized bar resolution of $\mathbb{Z}$, see \cite{eilenberg1} for more details.  Let $G$ be a discrete group and let $A$ be a $\mathbb{Z}[G]$-module. Let $C^0(G,A)=A$, and let  $$C^n(G,A)=\{f:\underbrace{G\times\cdots \times G}_{n-times}\to A| f(x_1,\ldots,x_n)=0, \text{ if } x_i=1_G \text{ for some }i \},$$for $n \geq 1$.

Consider the cochain complex
\begin{equation*}\label{complex}
0 \longrightarrow C^0 (G, A) \stackrel{\delta_0}{\longrightarrow }
C^1 (G, A) \stackrel{\delta_1}{\longrightarrow }C^2 (G, A) \cdots C^{n} (G, A)
\stackrel{\delta_n}{\longrightarrow } C^{n+1} (G, A) \cdots
\end{equation*} where
\begin{align*}
    \delta_n(f)(x_1,x_2,\ldots,x_{n+1})&=x_1\cdot f(x_2,\ldots,x_{n+1})\\
    &+\sum_{i=1}^n (-1)^{i} f(x_1,\ldots,x_{i-1},x_ix_{i+1},x_{i+2},\ldots,x_{n+1})\\
    &+ (-1)^{n+1}f(x_1,\ldots,x_{n}).
\end{align*}
We denote,  $Z^n(G,A):=\ker(\delta_n)$ ($n$-cocycles), $B^n(G,A):= \text{Im}(\delta_{n-1})$ ($n$-coboundaries) and
$$H^n(G,A):=Z^n(G,A)/B^n(G,A) \  \  (n\geq 1),$$  the cohomology of $G$ with coefficients in $A$.

\subsection{Eilenberg-MacLane cohomology theory of abelian groups}
Let $A$ be an abelian group. A space $X$ having only one nontrivial homotopy group $\pi_n(X)= A$
is called the Eilenberg-MacLane space $K(A, n)$. Such space can be constructed as a CW complex or using the  \textit{Dold-Kan correspondence} between chain complexes and simplicial abelian groups. If $A[n]$ is the chain complex which is $A$ in dimension $n$  and trivial elsewhere; the geometric realization of the corresponding simplicial abelian group is a $K(A,n)$ space.

The abelian cohomology theory
of the abelian group $M$ with coefficients in the abelian group $N$ is defined as $$H^{n}_{ab}(M,N):=\{\text{Homotopy classes }\  K(M,2)\to K(N,n+1)\}$$

In \cite{eilenberg1,eilenberg2} Eilenberg and MacLane defined a chain complex associated to any abelian group $M$ to
compute the abelian cohomology groups of the space $K(M, 2)$. 

We use the following notations for $X$, $Y$ any two groups:

\begin{itemize}
\item $X^p|Y^q=\{\bm{x}|\bm{y}=(x_1,\ldots,x_p|y_1,\ldots,y_q), x_i \in X, y_j\in Y\},$ \  $p,q\geq 0$.
\item ${\rm{Shuff}}(p,q)$ the set of  $(p,q)$-shuffles, i.e. an element in the symmetric group ${\mathbb{S}}_{p+q}$ such that $\lambda(i) <\lambda(j)$ whenever $1 \leq i < j \leq p $ or $p+1 \leq i < j \leq p+q$.
\item Any $\pi \in {\rm{Shuff}}(p,q)$ defines a map 
\begin{align}
\pi: X^{p+q}&\to X^{p+q}\\
(x_1,\ldots ,x_{p+q})&\mapsto (x_{\pi(1)}\ldots,x_{\pi(p+q)})
\end{align}
\end{itemize}

Let $M$ and $N$ be abelian groups. Define the abelian group $C^0_{ab}(M,N)=0$ and for $n>0$

$$C^n_{ab}(M,N)=\bigoplus_{p_1,\ldots,p_r\geq 1: r+\sum_{i=1}^rp_i=n+1}\Maps(M^{p_1}|\cdots |M^{p_r},N),$$
where $\Maps(M^{p_1}|\cdots |M^{p_r},N)$ denotes the abelian group of all maps from $M^{p_1}|\cdots |M^{p_r}$ to $N$.

The coboundary maps are defined as $$\partial:C^n_{ab}(M,N)\to C^{n+1}_{ab}(M,N)$$
\begin{align*}
    \partial (f)(\bm{x}^1|\bm{x}^2|\ldots|\bm{x}^{r})= & \sum_{\substack{1\leq i \leq r\\ 0\leq j\leq p_i}} (-1)^{\epsilon_{i-1}+j}f(\bm{x}^1|\ldots|d_j\bm{x}^i|\cdots|\bm{x}^r)\\
    &+ \sum_{\substack{1\leq i \leq r-1\\  \pi\in {\rm{Shuff}}(p_i,p_{i+1})}} (-1)^{\epsilon_{i}+\epsilon(\pi)}f(\bm{x}^1|\ldots|\pi (\bm{x}^i|\bm{x}^{i+1})|\cdots|\bm{x}^r)
\end{align*}
where 
\begin{align*}
d_j:M^{p_i}&\to M^{p_i-1}\\ (x_1,\ldots,x_{p_i})&\mapsto (x_1,\ldots,x_{i-1},x_ix_{i+1},x_{i+2},\ldots,x_{p_i})
\end{align*}are the face operators; $\epsilon_i=p_1+\cdots p_i+i$ and $\epsilon(\pi)$ is the sign of the shuffle $\pi$.

We denote,  $Z^n_{ab}(M,N):=\ker(\partial_n)$ (called abelian $n$-cocycles), $B^n_{ab}(M,N):= \operatorname{Im}(\partial_{n-1})$ (called abelian $n$-coboundaries) and
$$H^n_{ab}(M,N):=Z^n_{ab}(M,N)/B^n_{ab}(M,N) \  \  (n\geq 1),$$ the abelian cohomology of $M$ with coefficients in $N$.

Let us write the first cochains groups and their coboundaries.

\begin{itemize}
\item $C^0_{ab}(M,N)=0$,
\item $C^1_{ab}(M,N)=\Maps(M,N)$,
\item $C^2_{ab}(M,N)=\Maps(M^2,N) $,
\item $C^3_{ab}(M,N)=\Maps(M^3,N)\oplus \Maps(M|M,N)$
\item $C^4_{ab}(M,N)=\Maps(M^4,N)\oplus \Maps(M^2|M,N)\oplus \Maps(M|M^2,N)$.

\end{itemize}
Thus 
\begin{itemize}
\item Since  $C^0_{ab}(M,N)=0$,  $H^1_{ab}(M,N)=Z^1_{ab}(M,N)=\Hom(M,N)$.
\item For $f\in C^2_{ab}(M,N)$, we have  $$\partial(f)(x,y,z)= f(y,z)-f(xy,z)+f(y,z)-f(x,yz), \  \  \partial(f)(x|y)=f(x,y)-f(y,x).$$ Then $H^2_{ab}(M,N)\cong \operatorname{Ext}_{\mathbb{Z}}^1(M,N)$ the group of abelian extensions of $M$ by $N$.
\item Finally, for $(\omega, c)\in C^3_{ab}(M,N)$ we have 
\begin{itemize}
\item[(i)] $\partial(\omega)(x,y,z,t)=\omega(y,z,t)-\omega(x+y,z,t)+\omega(x,y+z,t)-\omega(x,y,z+t)+\omega(x,y,z)$,
\item[(ii)] $\partial(c)(x|y,z)=c(x|z)-c(x|y+z)+c(x|y)+w(x,y,z)-\omega(y,x,z)+\omega(y,z,x)$,
\item[(iii)]$\partial(c)(x,y|z)=c(y|z)-c(x+y|z)+c(x|z)-\omega(x,y,z)+\omega(x,z,y)-\omega(z,x,y).$
\end{itemize}
\end{itemize}

\section{Fusion algebras}
A fusion algebra is based on a finite set $A$ (where elements will be called  anyonic particles or simply particles). The elements in $A$ will be denoted by $a,b,c,\ldots .$

For every particle $a$ there exists a unique  anti-particle, that we denote by $\overline{a}$. There is a unique trivial “vacuum” particle denoted by $1$ (or sometimes 0). 

The fusion algebra has  \textit{fusion rules}

\begin{equation*}
    a\times b= \sum_{c}N_{ab}^cc
\end{equation*}
where $N_{ab}^c\in \mathbb{Z}^{\geq 0}$  that count the number of ways the particles $a$ and $b$ fuse into $c$. The fusion rules obey the following relations

\begin{itemize}
    \item associativity $(a\times b)\times c= a\times (b\times c)$,
    \item commutativity  $a\times b=b\times a$,
    \item the vacuum is the identity for the fusion product, $a\times 1=a,$
    \item the rule $a\mapsto \overline{a}$ defines an involution of the fusion rules, that is, $$\overline{1}=1,\  \ \overline{\overline{a}}=a,\  \ \overline{a}\times \overline{b}= \overline{a\times b},$$ where $$\overline{a\times b}= \sum_c N_{ab}^c \overline{c}.$$
    \item The fusion of $a$ with its antiparticle $\overline{a}$ contains the vacuum with multiplicity one, that is $$N_{a\overline{a}}^1=1.$$

\end{itemize}

A fusion algebra is called \textit{abelian} if $$\sum_c N_{ab}^c=1$$ for every $a$ and $b$. This is if the fusion of two particles $a \times b = c$, is again one of the particles in $A$. If $A$ is an abelian fusion algebra, then the fusion product defines a structure of abelian group on $A$ and conversely every finite abelian group defines a set of abelian fusion rules.

If we have a fusion algebra on the set $A$ with $n$ particles, we can assign to each particle $a$ the matrix $N_a$ whose entries are exactly $N_{ab}^c$ in the position $(b,c)$. This is an $n\times n$ integer matrix that contains all the information about the fusion rules of $a$. It satisfies the equation $$N_a N_b = \sum_c N_{ab}^c N_c.$$

\subsection{The Pentagon equation for abelian anyons}

Throughout this section, we will follow the notation of \cite{tambara1998tensor}, slightly modified for our purposes. For further reading on these topics we direct the reader to \cite{bernevig2015topological,bonderson2007non,kitaev2006anyons}.

Let $A$ be a fusion algebra. Assign to each fusion product a vector space $\FuSp{a}{b}{c}$ of dimension $N_{a,b}^c$.
If $N_{a,b}^c=0$ then $\FuSp{a}{b}{c}=0$. The vector spaces $\FuSp{a}{b}{c}$ are called the \textit{fusion spaces} of $A$. The fusion space takes in account the ways in which the anyons $a$ and $b$ can fuse together to give $c$.


Now, consider the fusion of the particles $a$, $b$ and $c$. The associativity of the fusion rules ensures that $(a \times b) \times c= a \times (b \times c)$, but with the fusion spaces there are two different objects that can do this. The first being
$$\bigoplus_{i\in A}\FuSp{a}{b}{i} \otimes \FuSp{i}{c}{d},$$ and the second being  $$\bigoplus_{i\in A} \FuSp{b}{c}{i} \otimes \FuSp{a}{i}{d}.$$

We would like a family of linear isomorphisms that takes in account the distinct ways of "associating" fusion spaces in this context, thus we have the following definition:

An $F$-matrix for a fusion algebra $A$ is a family of linear isomorphisms 

\begin{align*}
   {\Fmatrix{a}{b}{c}{d}} :\bigoplus_{i\in A}\FuSp{a}{b}{i} \otimes \FuSp{i}{c}{d} &\longrightarrow \bigoplus_{j\in A} \FuSp{b}{c}{j} \otimes \FuSp{a}{j}{d}
\end{align*}

which satisfies the pentagon equation:

\begin{equation*}
    \xymatrix{
    \bigoplus_{i,j} \FuSp{a}{b}{i} \FuSp{i}{c}{j} \FuSp{j}{d}{e} 
    \ar[r]^{\Fmatrix{a}{b}{c}{j}}
    \ar[d]^{\Fmatrix{i}{c}{d}{e}}&
    \bigoplus_{i,j} \FuSp{b}{c}{i} \FuSp{a}{i}{j} \FuSp{j}{d}{e} 
    \ar[r]^{\Fmatrix{a}{i}{d}{e}}&
    \bigoplus_{i,j} \FuSp{b}{c}{i} \FuSp{i}{d}{j} \FuSp{a}{j}{e} 
    \ar[d]^{\Fmatrix{b}{c}{d}{j}}\\
    \bigoplus_{i,j} \FuSp{a}{b}{i} \FuSp{c}{d}{j} \FuSp{i}{j}{e} 
    \ar[r]^{\tau}&
    \bigoplus_{i,j} \FuSp{c}{d}{i} \FuSp{a}{b}{i} \FuSp{i}{j}{e} 
    \ar[r]^{\Fmatrix{a}{b}{j}{e}}&
    \bigoplus_{i,j} \FuSp{c}{d}{i} \FuSp{b}{i}{j} \FuSp{a}{j}{e}
    },
\end{equation*}
or simply
\begin{equation}\label{pentagonal.eq}
  \sum_{i,j\in A}\Fmatrix{b}{c}{d}{j}  \Fmatrix{a}{i}{d}{e}\Fmatrix{a}{b}{c}{j}= \sum_{i,j\in A}\Fmatrix{i}{c}{d}{e} \Fmatrix{a}{b}{j}{e}.
\end{equation}
 In the diagram above, 
 $$\tau:\bigoplus_{i,j} \FuSp{a}{b}{i} \FuSp{c}{d}{j} \longrightarrow \bigoplus_{i,j} \FuSp{c}{d}{i} \FuSp{a}{b}{i} $$
 is the operator that swaps the components of $\FuSp{a}{b}{i}$ and $\FuSp{c}{d}{j}$.
 We also omit the tensor products and identity operators for simplicity.


We want that any transformation through the $F$-matrix starting and ending in the same spaces to be the same. Equation \eqref{pentagonal.eq} ensures this. 


Let us assume that $A$ is an abelian fusion algebra. Then we must have that each fusion space is either one or zero dimensional and an $F$-matrix for  $A$ is determined by a family of scalars $$\left\{\omega(a,b,c):=\Fmatrix{a}{b}{c}{d}\in \mathbb{C}^*\right\}_{a,b,c\in A}$$ such that 
\begin{equation}\label{3-cociclo}
\omega(a_1a_2,a_3,a_4)\omega(a_1,a_2,a_3a_4)=\omega(a_1,a_2,a_3)\omega(a_1,a_2a_3,a_4)\omega(a_2,a_3,a_4),    
\end{equation}
for all $a_1,a_2,a_3,a_4\in A$.

A function $\omega: A\times A\times A\to \Uone$ satisfying equation \eqref{3-cociclo} is just a standard 3-cocycle. Thus, the set of all solutions of the pentagon equation of an abelian fusion algebra is exactly $Z^3(A,\Uone)$.

A gauge transformation between two solution of pentagon equation $\omega, \omega' \in Z^3(A,\Uone)$ is determined by a family of non zero scalars $\{u(a,b)\}_{a,b\in A}$ such that $$\omega'(a,b,c)=\frac{u(ab,c)u(a,b)}{u(a,bc)u(b,c)}\omega(a,b,c),$$ for all $a,b,c\in A$. Thus, the set of gauge equivalence classes of solution of the pentagon equation is the $H^3(A,\Uone)$.

\subsection{The hexagon equation}
In this section we will assume that $A$ is an abelian group and $\omega\in Z^3(A,\Uone)$ is a 3-cocycle.

In the previous section, we extended the associativity of the fusion rules to the associativity of the fusion spaces through a family of linear operators called $F$-matrix. Now, we want to extend the commutativity as well. 

In order to do this, we need a family of unitary operators
$$R_{a,b}^c: \FuSp{a}{b}{c} \to \FuSp{b}{a}{c}$$ that satisfy $$R_{a,1}^a= \operatorname{Id} = R_{1,a}^a$$ and the hexagon equations
\begin{align}
\sum_{i,j,k} R_{a,c}^i \Fmatrix{b}{a}{c}{j} R_{a,b}^k &=  \sum_{i,j,k}\Fmatrix{a}{c}{b}{i} R_{b,c}^j \Fmatrix{a}{b}{c}{k} \\
\sum_{i,j,k} (R_{a,c}^i)^{-1} \Fmatrix{b}{a}{c}{j} (R_{a,b}^k)^{-1} &=  \sum_{i,j,k}\Fmatrix{a}{c}{b}{i} (R_{b,c}^j)^{-1} \Fmatrix{a}{b}{c}{k} .
\end{align} 

We will call such family an \textit{$R$-matrix}, or a \textit{braiding}, for $A$.
As before, these equations imply that any transformation within the $R$ and the $F$-matrices are independent of the path.

In the case where $A$ is an abelian theory with an associated 3-cocycle $\omega$, a braiding is determined by a family of scalars $\{c_{a,b}\}_{a,b\in A}$ that satisfy the equations
\begin{align*}
\frac{\omega(b,a,c)}{\omega(a,b,c) \omega(b,c,a) }&= \frac{\cc(a,bc)}{\cc(a,b) \cc(a,c)}\\
\frac{\omega(a,b,c) \omega(c,a,b) }{\omega(a,c,b)}&= \frac{\cc(ab,c)}{\cc(a,c) \cc(b,c)}.
\end{align*}

Thus, $(\omega,\cc)$ is an \textit{abelian $3$-cocycle}. The solutions of the hexagon up to gauge equivalence is the group $H^3_{ab}(A,\Uone)$.

\section{Computing $H^3_{ab}(M,N)$}

\subsection{Quadratic forms  and $H^3_{ab}(A,B)$}

Let  $A$ and $B$ be abelian groups. A quadratic form from $A$ to $B$ is a function $\gamma: A\to B$ such that

\begin{align}
    \gamma(a)&=\gamma(-a) \label{Gamma 1}\\
    \gamma(a+b+c)-\gamma(b+c)-\gamma(a+c)&-\gamma(a+b)+\gamma(a)+\gamma(b)+\gamma(c)=0,\label{Gamma 2}
\end{align}for any $a,b,c\in A$. A map $\gamma:A\to \B$ such that $\gamma(a)=\gamma(-a)$ satisfies $\eqref{Gamma 2}$ if and only if the map
\begin{align*}
 b_\gamma :A\times A&\to B\\
(a_1,a_2)&\mapsto \gamma(a_1+a_2)- \gamma(a_1)-\gamma(a_2)
\end{align*}is a symmetric bilinear form. It follows  by induction that $\gamma(na)=n^2\gamma(a)$ for any positive integer $n$.

We will denote by $\cuad(A,B)$, the group of all quadratic forms from $A$ to $B$. Eilenberg and MacLane proved in  \cite[Theorem 26.1]{eilenberg2} that for any two abelian groups $A$, $B$, the map
\begin{align}
\operatorname{Tr}:H^3_{ab}(A,B)&\to \cuad(A,B) \label{Traza}\\    (\omega,c)&\mapsto [a\mapsto c(a,a)]\notag
\end{align}is a group isomorphism.

If $A$ is a finite abelian group, the group $\cuad(A,\Rmod)$ can be computed using the following results. 
\begin{proposition}\label{quad func cyclic}
If $n$ is odd, then $\cuad(\mathbb{Z}_n,\Rmod)$ is a cyclic group of order $n$, with generator given by 
\begin{align*}
    q_n: \mathbb{Z}_n&\to \Rmod\\
     m&\mapsto m^2/n.
\end{align*}
If $n$ is even $\cuad(\mathbb{Z}_n,\Rmod)$ is a cyclic group of order $2n$, with generator given by 
\begin{align*}
    q_{2n}: \mathbb{Z}_n &\to \Rmod\\
     m&\mapsto m^2/2n.
\end{align*}
\end{proposition}
\begin{proof}
Let $\gamma:\mathbb{Z}_n\to \Rmod$ be a quadratic form. Since  $\gamma(m)=m^2\gamma(1)$, the quadratic form is completely determined by $\gamma(1)\in \mathbb{Q}/\mathbb{Z}$. Since $q(n)=0$, $n^2q(1)=0$, and since $q(1)=q(-1)$, $2nq(1)=0$.

If $n$ is odd. Then $nq(1)=0$, so $q(1)\in\{1/n,2/n,\ldots, 0\}\subset \mathbb{Q}/\mathbb{Z}$ define all possible quadratic forms.  If $n$ is even, $q(1)\in\{1/2n,2/2n,\ldots, 0\}\subset \mathbb{Q}/\mathbb{Z}$ define the possible quadratic forms.
\end{proof}

\begin{remark}\label{construccion 3-cociclo de un ciclicoc par}
Let $n$ be  an even positive integer. An abelian 3-cocycle $(\omega,c)\in Z^3(\mathbb{Z}/n\mathbb{Z},\Rmod)$ representing the cohomology class of the quadratic form $q_{2n}$ is given by 
\begin{align*} c(a,b)=\frac{ab}{2n}, \  \  \     \  \  \omega(a,b,c)=  \begin{cases} \frac{a}{2}, \qquad &\text{if } b+c\geq n,\\
0. \quad  &\text{ other case.}
\end{cases}
\end{align*}
\end{remark}

\begin{proposition}\label{formas cuadra de una suma}
Let $A$ and $B$ be abelian group, then the map 
\begin{align*}
    T:\Hom(A\otimes B, \Rmod)&\oplus \cuad(A,\Rmod)\oplus \cuad(B,\Rmod) \to \cuad(A\oplus B,\Rmod)\\
    f\oplus \gamma_A\oplus  \gamma_B&\to [(a,b)\mapsto f(a\oplus b)+\gamma_A(a)+\gamma_B(b)],
\end{align*}is a group isomorphisms.
\end{proposition}
\begin{proof}
We will see that 
\begin{align*}
    W: \cuad(A\oplus B,\Rmod) &\to \Hom(A\otimes B, \Rmod)\oplus \cuad(A,\Rmod)\oplus \cuad(B,\Rmod)\\
    \gamma &\mapsto \gamma_A+\gamma_B+b_\gamma|_{(A\oplus 0)\times (0)\oplus B},
\end{align*}is the inverse of $T$.
In fact,
\begin{align*}
T\circ W (\gamma)(a\oplus b)&= \gamma(a)+\gamma(b)+(\gamma(a\otimes b)-\gamma(a)-\gamma(b))\\
&=\gamma(a\oplus b),
\end{align*}
and 

\begin{align*}
W\circ T (f\oplus \gamma_A\oplus  \gamma_B)(a_1\otimes b_1\oplus a_2\oplus b_2)&= b_{T((f\oplus \gamma_A\oplus  \gamma_B)}(a_1\otimes b_1)\oplus \gamma_A(a_2)\oplus \gamma_B(b_2)\\
&= b_{T((f\oplus 0\oplus  0)}(a_1\otimes b_1)\oplus \gamma_A(a_2)\oplus \gamma_B(b_2)\\
&= f(a_1\otimes b_1)\oplus \gamma_A(a_2)\oplus \gamma_B(b_2).
\end{align*}
\end{proof}

\begin{corollary}\label{conteo}
If $A$ is a finite abelian group, then $$|\cuad(A,\Rmod)|=|A/2A| |S^2(A)|.$$
\end{corollary}
\begin{proof}
Recall that $S^2(A\oplus B)\cong S^2(A)\oplus S^2(B)\oplus A\otimes B$ for any pair of abelian groups. In particular, 
\begin{equation}\label{potencia simetirca}
|S^2(A\oplus B)|= |S^2(A)||S^2(B)||A\otimes B|.
\end{equation}
If $A=B\oplus C$, by Proposition \ref{formas cuadra de una suma} we have 

\begin{align*}
    |\cuad(A,\Rmod)|&=|B/2B| |S^2(B)|||C/2C||S^2(C)||B\otimes C|\\
    &= (|B/2B||C/2C|)(|S^2(B)||S^2(B)||B\otimes C)|)\\
    &=|A/2A||S^2(A)|.
\end{align*}
\end{proof}

\subsection{A double complex for an abelian group.}

To describe the obstruction to the existence of a solution of the hexagon equation of a 3-cocycle $\omega \in Z^3(A,U(1))$, in this section we will define a double complex associated to an abelian group.

Let $A$ and $N$ be abelian groups.
We define a double complex by $D^{p,q}(A,N)=0$ if $p$ or $q$ are zero and$$D^{p,q}(A ,N):= \Maps(A^p | A^q; N),\  \  p,q > 0$$
 with horizontal and vertical differentials the standard differentials, that is,
  $$\delta_h : D^{p,q}(A, N)=C^p(A,C^q(A,N)) \to
D^{p+1,q}(A, N)=C^{p+1}(A,C^q(A,N))$$ and 
 $$\delta_v : D^{p,q}(A, N)=C^q(A,C^p(A,N)) \to
D^{p,q+1}(A,N)=C^{q+1}(A,C^p(A,N))$$ 
defined by the equations
\begin{align*}(\delta_hF)(g_1,..., g_{p+1} || k_1,...,k_q)
= F(g_2,&...,g_{p+1}|| k_1,...,k_q)\\ & + \sum_{i=1}^{p}(-1)^i F(g_1,...,g_ig_{i+1},
..,g_{p+1}|| k_1,...,k_q )\\ &+ (-1)^{p+1} F(g_1,...,g_p||
k_1,...,k_q ) \\
(\delta_vF)(g_1,...,g_p|| k_1, ..., k_{q+1} ) = 
F(g_1,&...,g_{p}|| k_2 ,..., k_{q+1}) \\ &+
\sum_{j=1}^{q}(-1)^j F(g_1,...,g_p|| k_1, ...,k_jk_{j+1},..., k_{q+1} )\\
&+ (-1)^{q+1} F(g_1,...,g_p|| k_1, ..., k_{q}).
\end{align*}

For future reference it will be useful to describe the equations that define a
2-cocycle and the coboundary of a 1-cochains:

\begin{itemize}
\item $\operatorname{Tot}^0(D^{*,*}(A,N))=\operatorname{Tot}^1(D^{*,*}(A,N))=0$,
\item $\operatorname{Tot}^2(D^{*,*}(A,N))=\Maps(A|A,N)$,
\item $\operatorname{Tot}^3(D^{*,*}(A,N))=\Maps(A|A^2,N)\oplus \Maps(A^2|A,N)$,
\end{itemize}
Thus,

\begin{itemize}
\item $H^2(\operatorname{Tot}^*(D^{*,*}(A,N)))=\Hom(A^{\otimes 2},N)$ the abelian group of all bicharacters from $A$ to $N$.
\item for $f\in \operatorname{Tot}^2(D^{*,*}(A,N)),$
\begin{align*}
\delta_h(f)(x,y||z)=f(y||z)-f(x+y||z)+f(x||z)\\
\delta_v(f)(x||y,z)=f(x||z)-f(x||y+z)+f(x||y)
\end{align*}
\end{itemize}
Let us describe the elements $$(\alpha, \beta )\in Z^3(\operatorname{Tot}^*(D^{*,*}(A,N))),$$
$\alpha\in C^1(A,Z^2(A,N))$, that is $\alpha: A|A^2\to N$ such that $$\alpha(x;a,b)+\alpha(x;a+b,c)=\alpha(x;a,b+c)+\alpha(x;b,c)$$ for all $x, a,b, c\in A$,

$\beta\in C^1(A,Z^2(A,N))$, that is 
$\beta: A^2|A\to N$ such that $$\beta(x+y,z;a)+\beta(x,y;a)=\beta(x,y+z,a)+\beta(y,z;a)$$
for all $x,y,z\in A, a\in A$, and $ \delta_h(\alpha)=-\delta_v(\beta)$, that is, $$\alpha(x;a,b)+\alpha(y;a,b)-\alpha(x+y;a,b)=\beta(x,y;a+b)-\beta(x,y;a)-\beta(x,y;b),$$for all $x,y\in A, a,b\in B$.

\subsection{Obstruction}

Consider  the group homomorphism 
\[
\tau_n: H^n(A,N)\to H^n(\operatorname{Tot}^*(D^{*,*}(A,N)))
\]
induced by the cochain map 
\begin{align*}
\tau: C^*(A,N) &\to C^n(\operatorname{Tot}^*(D^{*,*}(A,N)))\\ 
\alpha &\mapsto \oplus_{p=1}^{n-1}\alpha_p,
\end{align*}
where $\alpha_p\in \Maps(A^{p}| A^{n-p},N)$ is defined by $$\alpha_p(a_1,\ldots,a_p|a_{p+1},\ldots,a_n)= \sum_{\pi \in \rm{Shuff}(p,n-p)} (-1)^{\epsilon(\pi)}\alpha(a_{\lambda(1)},\ldots, a_{\pi(n)}).$$
For every $n\in \mathbb Z^{\geq 2}$, we define the suspension  homomorphism from
\begin{align*}
s_n:H^n_{ab}(A,N)&\to H^n(A,N)\\
\oplus_{p_1,\ldots,p_r\geq 1: r+\sum_{i=1}^rp_i=n+1} \alpha_{p_1,\ldots,p_r} &\mapsto \alpha_n.
\end{align*}

The group homomorphism
\begin{align*}
 H^2(\operatorname{Tot}^*(D^{*,*}(A,N)))=\Hom(A^{\otimes 2},N)&\to Z^3_{ab}(A,N)\\
c &\mapsto (0,c),
\end{align*}
induces a group homomorphism $\iota: H^2(\operatorname{Tot}^*(D^{*,*}(A,N))\to H^3_{ab}(A,N)$.

The following result shows that the shuffle homomorphism can be interpreted as the obstruction to the hexagon equation.

\begin{theorem}
Let $A$ and $N$ be a abelian groups. Then, the sequence

\begin{equation*}\label{sequence}
\begin{tikzcd}
0\ar{r}  & H^2_{ab}(A,N)\ar{r}{s_2}  & H^2(A,N) \ar{r}{\tau_2} & H^2(\operatorname{Tot}^*(D^{*,*}(A,N))  \ar{lld}{\iota}& \\& H^3_{ab}(A,N)  \ar{r}{s_3}
    & H^3(A,N) \ar{r}{\tau_3} & H^3(\operatorname{Tot}^*(D^{*,*}(A,N))) .
\end{tikzcd}
\end{equation*}
is exact. 
\end{theorem}
\begin{proof}
The shuffle homomorphism $\tau_2:H^2(A,N) \to H^2(\operatorname{Tot}^*(D^{*,*}(A,N))$ is given  by  $\tau(\alpha)(x,y)=\alpha(x,y)-\alpha(y,x)$, thus it is clear that the sequence is exact in $H^2(A,N)$. 

An abelian 3-cocycle  is in the kernel of the suspension map if it is cohomologous to an abelian 3-cocycle of the form $(0,c)$. But then $c\in \Hom(A^{\ot 2},N)=  H^2(\operatorname{Tot}^*(D^{*,*}(A,N))$, hence the sequence is exact in $H^3_{ab}(A,N)$. Finally, 
If $\omega \in Z^3(A,N)$, then $\tau(\omega)= (\alpha_\omega,\beta_\omega)$, where 

\begin{align*}
\alpha_\omega(x|y,z)&=\omega(x,y,z)-\omega(y,x,z)+\omega(y,z,x)\\
\beta_\omega(x,y|z)&= \omega(x,y,z)-\omega(x,z,y)+\omega(z,x,y).
\end{align*}
Thus, $[(\alpha_\omega,\beta_\omega)]=0$, if and only if there is  $c:A\times A\to N$ such that $$\delta_v(c)=\alpha_\omega, \ \  \  -\delta_h(c)=\beta_\omega,$$that is, $[\tau(\omega)]=0$ if and only if there is  $c:A\times A\to N$  such that $(\omega,c)\in Z^3_{ab}(A,N)$. Thus, the sequence is exact in $H^3(A,N)$.
\end{proof}

\begin{corollary}[Total obstruction]
A gauge class of a solution of the pentagon equation $\omega \in H^3(A,\Rmod)$ admits a solution of the hexagon equation if and only if $\tau(\omega)=0$ in $H^3(\operatorname{Tot}^*(D^{*,*}(A,\Rmod)))$. 
\end{corollary}
\qed

\begin{proposition}\label{kernel suspension}
Let $A$ be a finite  abelian group. 
\begin{enumerate}
\item $$\operatorname{ker}(s_3)\cong \widehat{S^2(A)}.$$
\item Under the isomorphism $\operatorname{Tr}: H^3_{ab}(A,\Rmod)\to \cuad(A,\Rmod)$ (see \eqref{Traza}), $\operatorname{ker}(s_3)$ corresponds to the subgroup $$\cuad_0(A,\Rmod)=\{q\in \cuad(A,\Rmod): o(a)q(a)=0, \forall a\in A\},$$
where $o(a)$ denotes the order of $a\in A$.
\end{enumerate}

\end{proposition}
\begin{proof}
Since $\Rmod$ is divisible and $A$ is finite, the group $H^2_{ab}(A,\Rmod)=\operatorname{Ext}_{\mathbb{Z}}^1(A,\Rmod)$ is null. Thus, by Theorem \ref{sequence} we have an exact sequence

$$0\to H^2(A,\Rmod)\overset{\tau}\to \Hom(A^{\otimes 2},\Rmod)\to \operatorname{ker}(s_3)\to 0.$$
The image of $\tau$ is $\Hom(\wedge^2 A,\Rmod)$, hence  
\begin{align*}
    \operatorname{ker}(s_3) &\cong \Hom(A^{\otimes 2},\Rmod)/\Hom(\wedge^2 A,\Rmod)\\
     &\cong \Hom(A^{\otimes 2}/\wedge^2 A,\Rmod)\\ &\cong \Hom(S^2(A),\Rmod),
\end{align*}where the last isomorphism is defined using the exact sequence 
\begin{align*}
    0\to \wedge^2A&\to A^{\otimes 2}\to S^2A\to 0. 
\end{align*}
Now we will prove the second part. If $A$ is cyclic the proposition follows by Remark \ref{construccion 3-cociclo de un ciclicoc par}. The general case follows from Proposition \ref{formas cuadra de una suma}, since the image of $\Hom(A\otimes B,\Rmod)$ by $T$ lies in $\cuad_0(A\oplus B,\Rmod)$.
\end{proof}

We will denote by $\mu_2=\{1,-1\}\subset \Uone \cong \Rmod$. 
\begin{theorem}\label{imagen de suspension}
Let $A $ be an abelian finite group. The canonical projection $\pi:A\to A/2A$ induces an isomorphism between the images of the respective suspension maps of $H^3_{ab}(A,\Rmod)$ and $H^3_{ab}(A/2A,\Rmod)$. 

Moreover, for  an elementary abelian 2-group  $(\mathbb{Z}/2\mathbb{Z})^{\oplus n}$,   $$\operatorname{Im}(s_3)\cong H^3(\mathbb{Z}/2\mathbb{Z},\mu_2)^{\oplus n}= (\mathbb{Z}/2\mathbb{Z})^{\oplus n}.$$

\end{theorem}
\begin{proof}
By Corollary \ref{conteo}, Proposition \ref{kernel suspension} and Theorem \ref{sequence}, we have that $|\operatorname{Im}(s_3)|=|A/2A|$. In particular the size of 
the image of the suspension maps of $H^3_{ab}(A,\Rmod)$ and $H^3_{ab}(A/2A,\Rmod)$ are equal.  Since $\pi^*$ is an inyective map between the image of the suspension maps it is an isomorphisms.

Let $(\mathbb{Z}/2\mathbb{Z})^{\oplus n}$ be  an  elementary abelian 2-group.
Recall that $H^3(\mathbb{Z}/2\mathbb{Z},\Rmod)=H^3(\mathbb{Z}/2\mathbb{Z},\mu_2)\cong \mathbb{Z}/2\mathbb{Z} $. Using the Remark \ref{construccion 3-cociclo de un ciclicoc par} we have an injective group homomorphisms   $H^3(\mathbb{Z}/2\mathbb{Z},\mu_2)^{\oplus n}\to \operatorname{Im}(s_3)$, which is an isomorphism because both groups have the same order.
\end{proof}

\begin{corollary}\label{corolario secesion}
For any abelian group we have an exact sequence 
\begin{equation}\label{susesion dos}
0\to S^2(\widehat{A})\to H^3_{ab}(A,\Rmod)\to A/2A\to 0.
\end{equation}
\end{corollary}
\qed

\begin{remark}
 A related result is established by Mason and
Ng in \cite[Lemma 6.2]{Mason-Ng}.
\end{remark}

\subsection{Explicit abelian 3-cocycles}

Abelian 3-cocycles for the group $\mathbb{Z}/2\mathbb{Z}\oplus \mathbb{Z}/2\mathbb{Z}$ were study in detail in \cite{bulacu2011braided}, for $\mathbb{Z}/m\mathbb{Z}\oplus \mathbb{Z}/n\mathbb{Z}$ in \cite{huang2014braided} and in full generality in \cite{quinn1999group}.

Let $A$ be a finite abelian group  and  $A=A_1\oplus A_2$ the canonical decomposition, where $A_1$ has order a power of $2$ and $A_2$ has odd order. Since $H^3_{ab}(A,\Rmod)=H^3_{ab}(A_1,\Rmod)\oplus H^3_{ab}(A_2,\Rmod)$, and $H^3_{ab}(A_2,\Rmod)\cong S^2(\widehat{A_2}),$ the problem of a general description can be divided in the case of group of odd order and the case of abelian 2-groups. 

\subsubsection{Case of $A$ an odd abelian group}
If $A$ is an odd abelian group then map $A\to A, a\mapsto 2a$ is a group automorphism of $A$. Hence, given $q\in \cuad(A,\Rmod)$ the symmetric bilinear form $c:=\frac{1}{2}b_q\in \Hom(A^{\otimes 2} ,\Rmod)$, defines an abelian 3-cocycle $(0,c)\in Z^3_{ab}(A,\Rmod)$ such that $\operatorname{Tr}(c)=q$.

\subsubsection{Case of $A$ an abelian 2-group}

Let $A=\bigoplus_{i=1}^n \mathbb{Z}/2^{m_i}\mathbb{Z}$. Then by Corollary \ref{corolario secesion} and Proposition \ref{kernel suspension} we have a 
commutative diagram 
\begin{equation} \label{exacta su}
\begin{tikzcd}
0 \ar{r}& \ar{d} \ar{r}\Hom(S^2(A),\Rmod) &\ar{d}{\operatorname{Tr}} H^3_{ab}(A,\Rmod)\ar{r}{s_3}& (\mathbb{Z}/2\mathbb{Z})^{\oplus n}\ar{d}{=}\ar{r} & 0\\
0\ar{r}& \ar{r}\cuad_0(A,\Rmod) &\ar{r}{\pi} \cuad(A,\Rmod) &\ar{r} (\mathbb{Z}/2\mathbb{Z})^{\oplus n}\ar{r} & 0
\end{tikzcd}
\end{equation}
where the horizontal sequences are exact and the vertical morphisms are isomorphisms.

Let $q\in \cuad_0(A,\Rmod)$, then $c\in \Hom(A^{\otimes 2},\Rmod)$ defined by  

$$c(\vec{x},\vec{y})= \sum_ix_iy_iq(\vec{e_i})+\sum_{i<j}x_iy_jb_q(\vec{e_i},\vec{e_j})$$
is a such that $(0,c)\in Z^3_{ab}(A,\Rmod)$ represents $q$.

A set-theoretical section $j:  (\mathbb{Z}/2\mathbb{Z})^{\oplus n} \to \cuad(A,\Rmod) $ of $\pi$ in the exact sequence \eqref{exacta su} is defined easily as follows 

$$j(\vec{y})(\vec{x})=\sum_i y_ix_i^2/2^{m_i+1}.$$
Abelian 3-cocycles representing the $j(\vec{e_j})'s$ are constructed as the pullback by the projection $\pi_j: \bigoplus_{i=1}^n \mathbb{Z}/2^{m_i}\to \mathbb{Z}/2^{m_j}$ of the abelian 3-cocycle $(w,c)\in Z^3(\mathbb{Z}/2^{m_j},\Rmod)$ defined in Remark \ref{construccion 3-cociclo de un ciclicoc par}.

As a consequence of the previous discussion we have the following result.
\begin{proposition}\label{Prop square trivial}
Let $A$ be an abelian group. Every abelian 3-cohomology class has a representative 3-cocycle $(\omega,c)\in Z^3_{ab}(A,\Rmod)$ where $\omega(a,b,c)\in \{0,\frac{1}{2}\}\subset \Rmod$.
\end{proposition}
\qed
\begin{remark}
Proposition \ref{Prop square trivial} implies that the cohomology class of the square power of an abelian 3-cocycle is zero. This result was  established  in \cite[Lemma 4.4 (ii)]{Ng-Total}. 
\end{remark}

\subsection{Partial obstructions}
Since the obstruction of the existence of a solution of the hexagon equation is an element in the total cohomology of a double complex, we can analyze the obstruction by partial obstructions as follows.
\begin{proposition}[Partial obstruction 1]\label{obstru 1}
Let $A$ and $N$ be abelian groups and $\omega \in Z^3(A,N)$. If  $\tau_3(\omega)=(\alpha_\omega,\beta_\omega)$, then the cohomology class of  $\alpha_\omega(a|-,-)\in Z^2(A,N)$ only depends on the cohomology class of $\omega$. If $\omega$ is in the image of the suspension map, then $$0=[\alpha_\omega(a|-,-)]\in H^2(A,N)$$
for all $a\in A$.
\end{proposition}
\begin{proof}
For $(\omega,c) \in H^3(A,N)$ we have that $\delta_v(c)=\alpha_\omega$, then $[\alpha_\omega(a|-,-)]=0$
for all $a\in A$.

Let $u:A\times A\to N$, and $$w'(a,b,c)=w(a,b,c)+ u(a+b,c)+u(a,b)-u(a,b+c)-u(b,c).$$ Then $$\alpha_{\omega'}(a|c,d)=\alpha_w(a|b,c)+l_a(b)+l_a(c)-l_a(b+c),$$ where $l_a(b)=u(a,b)-u(b,a)$.
\end{proof}

Assume that $[\alpha_\omega(a|-,-)]=0$
for all $a\in A$. Thus, there exists $\eta\in C(A|A, N)$ such that $\delta_v(\eta)=\alpha_\omega$. We have that  $(0,\delta_h(\eta)+\beta_\omega) \in Z^3(\operatorname{Tot}^*(D^{*,*}(A,N)))$, thus 
$$\theta(\omega,\eta):=\delta_h(\eta)+\beta_\omega \in Z^2_{ab}(A,\Hom(A,N)).$$ 
In fact, 

\begin{align*}
    \delta_v(\delta_h(\eta)+\beta_\omega)&= \delta_v(\delta_h(\eta))+\delta_v(\beta_\omega)\\
    &=\delta_h(\delta_v(\eta))+\delta_v(\beta_\omega)\\
    &=\delta_h(\alpha_\omega)-\delta_h(\alpha_\omega)=0,\\
\end{align*}
that is, $\theta(\omega,\eta)(a,b,x+y)=\theta(\omega,\eta)(a,b,x)+\theta(\omega,\eta)(a,b,y)$.

The cohomology of $\theta(\omega,\eta)$ does not depend on the choice of $\eta$. In fact, if $\eta\in C(A|A, N)$ such that $\delta_v(\eta)=\alpha_\omega$, then 
$\mu:=\eta-\eta'\in C^1(A,\Hom(A,U(1)))$ and 
$$\theta(\omega,\eta)-\theta(\omega,\eta')=\delta_h(\eta-\eta')=\delta_h(\mu).$$Hence, we have defined a second obstruction  
\begin{align*}
\theta(\omega)\in \operatorname{Ext}^1_{\mathbb{Z}}(A,\Hom(A,N).
\end{align*}

\begin{corollary}[Obstruction 2]
Let $\omega \in Z^3(A,N)$, such that $0=[\alpha_\omega(a|-,-)]\in Z^2(A,N)$
for all $a\in A$. Then there exists $c:A\times A\to N$ such that $(\omega,c)\in Z^3_{ab}(A,N)$ if and only if $\theta(\omega)=0$. 
\end{corollary}
\begin{proof}
If $(\omega,c)\in Z^3_{ab}(A,N)$, then $\tau(\omega)=0$, that implies  
$\theta(\omega)=0$. Now, let $\omega \in Z^3(A,N)$  and $\eta\in C(A|A, N)$ such that 
$\delta_v(\eta)=\alpha_\omega$, that is,  
$$\theta(\omega,\eta)=\delta_h(\eta)+\beta_\omega \in Z^2_{ab}(A,\Hom(A,N)).$$ 
If $[\theta(\omega,\eta)]=0$, there is $l:A\to \Hom(A,N)$ such that $\delta_h(l)=\delta_h(\eta)+\beta_\omega$. Thus $c:=\eta -l$ is such that $(\omega,c)\in Z^3_{ab}(A,N)$, since 

\begin{align*}
   \delta_v(c) =\delta_v(\eta)-\delta_v(l)=\alpha_\omega
\end{align*}
and 
\begin{align*}
    \delta_h(c)&=\delta_h(\eta)-(\delta_h(\eta)+\beta_\omega)\\
    &=-\beta_\omega.
\end{align*}

\end{proof}

\section{Abelian anyons}

In this last section we will present the classification of all possible \textit{prime} or \textit{indecomposable} abelian anyon theories.

\subsection{$S$ and $T$ matrices of abelian theories}

By an abelian theory we will mean a triple $(A,\omega,c)$, where $A$ is abelian group (or equivalently an abelian fusion rules) and $(\omega,c)\in Z^3(A,\Rmod)$ an abelian 3-cocycle (or equivalently a solution of the hexagon equation).

Let $(A,\omega,c)$ be an abelian theory. Recall that the associated quadratic form  $q:A\to \Rmod$ is defined by $q(a)=c(a,a)$. The  \textit{topological spin} of $a\in A$ is defined as the phase $$\theta_a=e^{2\pi i q(a)},$$thus, the topological spin is exactly the associated quadratic form and by the Eilenberg and MacLane theorem  \cite[Theorem 26.1]{eilenberg2} it determines up to gauge equivalence the abelian theory.

Recall that the symmetric bilinear form $$b_q:A\times A\to \Rmod$$ associated to the quadratic for  $q:A\to \Rmod$ is defined by $$b_q(a,b)=q(a+b)-q(a)-q(b).$$ Since $q(a)=c(a,a)$, we also have that $$b_q(a,b)=c(a,b)-c(b,a)$$for all $a,b \in A$. The map $$H^3_{ab}(A,\Rmod)\overset{b}\to \Hom(S^2(A))$$
that associates the symmetric bilinear form $b_q$ to an abelian 3-cocycle is a group homomorphism.
\begin{definition}
An anyonic abelian theory is an abelian theory $(A,\omega,c)$ such that  one of the following equivalent conditions holds:
\begin{enumerate}
\item The $S$-matrix $S_{ab}= |A|^{-1/2}e^{2\pi i b_q(a,b)}$ is non-singular.
\item The symmetric bilinear form $b_q$ is non-degenerated.
\end{enumerate}
\end{definition}
We will say that an abelian theory $(A,\omega,c)$ is \textit{symmetric} if its $b_q$ is trivial, or equivalently, if $-c(a,b)=c(b,a)$ for all $a,b\in A$. We will denote by $H^3_{s}(A,\Rmod)\subset H^3_{ab}(A,\Rmod)$ the subgroup of all equivalence class of symmetric abelian 3-cocycles.

The $T$-matrix of an abelian  anyonic theory is the diagonal matrix of the topological spins, that is, $$T_{ab}=\delta_{a,b}e^{2\pi i q(a)}.$$ Thus, for abelian anyons the $T$-matrix completely determines the theory. On the contrary, the $S$-matrix does not always  determine the theory, however the following result result say that two abelian anyons with the same $S$-matrix  only differ by a symmetric abelian 3-cocycle and their $T$-matrices by a linear character $\chi:A\to \{1,-1\}$.
\begin{proposition}
Let $A$ be an abelian group. Then the diagram
\begin{equation*} 
\begin{tikzcd}
0 \ar{r}& \ar{d}{\operatorname{Tr}} \ar{r}H^3_{s}(A,\Rmod) &\ar{d}{\operatorname{Tr}} H^3_{ab}(A,\Rmod)\ar{r}{b}& \Hom(S^2(A),\Rmod)\ar{d}{=}\ar{r} & 0\\
0\ar{r}& \ar{r}\Hom(A,\frac{1}{2}\mathbb{Z}/\mathbb{Z}) &\ar{r}{b} \cuad(A,\Rmod) &\ar{r} \Hom(S^2(A),\Rmod)\ar{r} & 0
\end{tikzcd}
\end{equation*}commutes, the vertical morphisms are isomorphisms and  the horizontal sequences are exact. 
\end{proposition}
\begin{proof}
Clearly the kernel of $b: \cuad(A,\Rmod)\to \Hom(S^2(A),\Rmod)$
is $\Hom(A,\frac{1}{2}\mathbb{Z}/\mathbb{Z})$. Thus, by Corollary \ref{conteo} the sequence $$0\to \Hom(A,\frac{1}{2}\mathbb{Z}/\mathbb{Z})\to\cuad(A,\Rmod)\to \Hom(S^2(A),\Rmod)\to 0,$$is exact.
\end{proof}
\begin{remark}
A related result was established  in \cite[Lemma 6.2 (ii), (iii)]{Mason-Ng}.
\end{remark}

\subsection{Prime abelian anyons}

If $(A,\omega,c)$ and $(A',\omega',c')$ are abelian anyons theory, their direct sum is defined as the anyon theory $(A\oplus A',\omega\times \omega',c\times c')$, where $\omega\times \omega'((a,a'),(b,b'),(c,c'))=\omega(a,b,c)\omega'(a',b',c')$ and similarly for $c\times c'$.

We will say that an anyon theory $(A,\omega,c)$ is \textit{prime} if for any non trivial subgroup $B\subset A$, the restriction of the associated bilinear form $b_q$ is degenerated. 

Two abelian theories $(A,\omega,c)$ and $(A,\omega,c)$ are called equivalents if there is a group isomorphism $f:A\to A'$ such that $(f^*(\omega),f^*(c)), (\omega,c)\in Z^3_{ab}(A,\Rmod)$ are cohomologous, or equivalently if $q'(f(a))=q(a)$ for all $a\in A$, where $q$ and $q'$ are the quadratic forms associated.

By  \cite[Theorem 4.4]{Mueguer} any abelian anyon theory is a direct sum of prime abelian anyon theory. Thus, the classification of  abelian anyons is reduced to the classification of prime abelian anyons. 

The Legendre symbol is a function of $a\in \mathbb{Z}^{>0}$ and a prime number $p$ defined as
\[
\left(\frac{a}{p}\right) = 
\begin{cases}
 1 & \text{ if } a \text{ is a quadratic residue modulo } p \text{ and } a \not\equiv 0\pmod{p}, \\
-1 & \text{ if } a \text{ is a quadratic non-residue modulo } p, \\
 0 & \text{ if } a \equiv 0 \pmod{p}.  
\end{cases}\]
Following the notation of \cite{MR839800}, we establish the classification of prime abelian anyons, that follows from results of Wall \cite{MR0156890} and Durfee \cite{MR0480333} about the classification of indecomposable non-degenerated quadratic forms on abelian groups. 

\begin{theorem}\label{clasificacion}
The following is the list of all equivalence  classes of prime abelian anyons theories:
\begin{enumerate}[(i)]
\item If $p\neq 2$ and $\epsilon=\pm 1$, 
$\omega_{p,k}^\epsilon$ denotes the abelian anyon with fusion rules given by $\mathbb{Z}/p^k\mathbb{Z}$ and abelian 3-cocycle 
$(0,c)$, where $c(x,y)=\frac{uxy}{p^k} \Mod{\mathbb{Z}},$  for some  $u\in \mathbb{Z}^{>0}$ with $(p,u)=1$ and $\left(\frac{2u}{p}\right)=\epsilon$.
\item If $\epsilon\in (\mathbb{Z}/8\mathbb{Z})^{\times}$, 
$\omega_{2,k}^\epsilon$ denotes the abelian anyon with fusion rules given by $\mathbb{Z}/2^k\mathbb{Z}$ and abelian 3-cocycle 

\begin{align*} c(x,y)=\frac{uxy}{2^{k+1}}\Mod{\mathbb{Z}}, \  \  \     \  \  \omega(x,y,z)=  \begin{cases} \frac{x}{2} \Mod{\mathbb{Z}}, \qquad &\text{if } y+z\geq 2^{k},\\
0 \Mod{\mathbb{Z}}. \quad  &\text{ otherwise.}
\end{cases}
\end{align*}
for some  $u\in \mathbb{Z}^{>0}$ with $ u \equiv \epsilon \Mod{8}$. The abelian anyons $w_{2,k}^1$ and $w_{2,k}^{-1}$ are defined for all $k\geq 1$ and $w_{2,k}^{5}$ and $w_{2,k}^{-5}$ for all $k\geq 2$.

\item $E_k$ denoted the abelian anyon with fusion rules given by $\mathbb{Z}/2^k\mathbb{Z}\oplus \mathbb{Z}/2^k\mathbb{Z} $ and abelian 3-cocycle $(0,c)$, where $c\in \Hom(\mathbb{Z}/2^k\mathbb{Z}\oplus \mathbb{Z}/2^k\mathbb{Z},\Rmod)$ is defined by

\begin{align*}
c(\vec{e}_i,\vec{e}_j)=\begin{cases} 0\Mod{\mathbb{Z}}, \qquad &\text{if } i=j, \text{or }  i=2, j=1.\\
2^{-k}\Mod{\mathbb{Z}}, \quad  &\text{if } i=1, j=2. 
\end{cases}
\end{align*}
\item $F_k$ denoted the abelian anyon with fusion rules given by $\mathbb{Z}/2^k\mathbb{Z}\oplus \mathbb{Z}/2^k\mathbb{Z} $ and abelian 3-cocycle $(0,c)$, where $c\in \Hom(\mathbb{Z}/2^k\mathbb{Z}\oplus \mathbb{Z}/2^k\mathbb{Z},\Rmod)$ is defined by

\begin{align*}
c(\vec{e}_i,\vec{e}_j)=\begin{cases}\  2^{-k}\Mod{\mathbb{Z}}, \qquad &\text{if } i=j,\\
\ 0\Mod{\mathbb{Z}}, \qquad &\text{if } i=2, j=1,\\
-2^{-k}\Mod{\mathbb{Z}}, \quad  &\text{if } i=1, j=2. 
\end{cases}
\end{align*}
\end{enumerate}
In the cases $E_k$ and $F_k$, we denote $\vec{e_1}=(1,0)$ and $\vec{e_2}=(0,1)$.
\end{theorem}
\begin{proof}
We recall from \cite{MR0156890} the basic structure of a non-degenerate finite qua\-dratic form over an abelian groups $G$. Let $(A,q)$ be a non-degenerate finite quadratic abelian group. The Sylow decomposition $G=\bigoplus_{p}A_p$ is an orthogonal direct sum decomposition with respect to the form $q$. Moreover, by the results of \cite{MR0156890} and \cite{MR0480333} each Sylow subgroup $A_p$, admits an orthogonal direct sum decomposition into indecomposable quadratic group of the following type:

\begin{enumerate}[(i)]
\item[(a)] If $p\neq 2$ and $\epsilon=\pm 1$, 
$\omega_{p,k}^\epsilon$ denotes the quadratic abelian group  $\mathbb{Z}/p^k\mathbb{Z}$ and quadratic form determined by  $q(1)=up^{-k}\Mod{\mathbb{Z}}$ for some  $u\in \mathbb{Z}^{>0}$ with $(p,u)=1$ and $\left(\frac{2u}{p}\right)=\epsilon$.
\item[(b)] If $\epsilon\in (\mathbb{Z}/8\mathbb{Z})^{\times}$, 
$\omega_{2,k}^\epsilon$ denotes the quadratic abelian group $\mathbb{Z}/2^k\mathbb{Z}$ and quadratic form determined by $q(1)=2^{-k-1}u \Mod{\mathbb{Z}}$
for some  $u\in \mathbb{Z}^{>0}$ with $ u \equiv \epsilon \Mod{8}$. The quadratic groups $w_{2,k}^1$ and $w_{2,k}^{-1}$ are defined for all $k\geq 1$ and $w_{2,k}^{5}$ and $w_{2,k}^{-5}$ for all $k\geq 2$.

\item[(c)] $E_k$ denoted the the quadratic abelian group $\mathbb{Z}/2^k\mathbb{Z}\oplus \mathbb{Z}/2^k\mathbb{Z} $  and quadratic form determined by $q((1,0))=q(0,1)=0 \Mod{\mathbb{Z}}$ and $q((1,1))=2^{-k}\Mod{\mathbb{Z}}$.
\item[(d)] $F_k$ denoted the the quadratic abelian group $\mathbb{Z}/2^k\mathbb{Z}\oplus \mathbb{Z}/2^k\mathbb{Z} $ and quadratic form determined by $q((1,0))=q(0,1)=q((1,1))=2^{-k}\Mod{\mathbb{Z}}$.
\end{enumerate}
Thus, to proof the theorem we only need to see that the abelian 3-cocycles defined in (i)-(iv) have the corresponding quadratic forms (a)-(d). In fact,

\begin{align*}
    \text{Case } \omega_{p, k}^{\epsilon}: && c(1,1)=up^{-k}.\\
\text{Case } \omega_{2, k}^{\epsilon}: && c(1,1)=u2^{-k-1}.\\
\text{Case } E_k: && c(\vec{e}_i,\vec{e}_i)=0, \ c((1,1),(1,1))=c(\vec{e}_1,\vec{e}_2)=2^{-k}.\\
\text{Case } E_k: && c(\vec{e}_i,\vec{e}_i)=2^{-k}, \ c((1,1),(1,1))=2c(\vec{e}_1,\vec{e}_1)+c(\vec{e}_1,\vec{e}_2)=2^{-k}.
\end{align*}
\end{proof}

We finish with the formulas of  the modular data of the prime abelian anyons theories (using the convention of identify $U(1)$ with $\mathbb{R}/\mathbb{Z}$):

\begin{enumerate}[(i)]
\item $\omega_{p,k}^\epsilon$  $p\neq 2$, $
\epsilon=\pm 1
$
\begin{align*}
    S_{a,b}= 2uabp^{-k}, && T_{a,a}= ua^2p^{-k}
\end{align*}
\item $\omega_{2,k}^\epsilon$ , $\epsilon \in(\mathbb{Z} / 8 \mathbb{Z})^{ \times}$
\begin{align*}
    S_{a,b}= uab2^{-k}, && T_{a,a}=ua^22^{-(k+1)}
\end{align*}

\item $E_k$ 
\begin{align*}
    S_{(a_1,a_2),(b_1,b_2)}=(a_1b_2+b_1a_2)2^{-k}, && T_{(a,b),(a,b)}=ab2^{-k}
\end{align*}

\item $F_k$ 
\begin{align*}
    S_{(a_1,a_2),(b_1,b_2)}&=(a_1b_1+a_2b_2)2^{-k+1}-(a_1b_2+a_2b_1)2^{-k}\\
    T_{(a,b),(a,b)}&=(a^2+b^2-ab)2^{-k}.
\end{align*}

\end{enumerate}
\providecommand{\bysame}{\leavevmode\hbox to3em{\hrulefill}\thinspace}
\providecommand{\MR}{\relax\ifhmode\unskip\space\fi MR }
\providecommand{\MRhref}[2]{%
  \href{http://www.ams.org/mathscinet-getitem?mr=#1}{#2}
}
\providecommand{\href}[2]{#2}


\end{document}